\documentclass[11pt]{amsart}
\usepackage[active]{srcltx}
\usepackage{amsmath, amsfonts,amsthm,times,graphics,color}
 \makeatletter
\renewcommand*\subjclass[2][2000]{%
  \def\@subjclass{#2}%
  \@ifundefined{subjclassname@#1}{%
    \ClassWarning{\@classname}{Unknown edition (#1) of Mathematics
      Subject Classification; using '1991'.}%
  }{%
    \@xp\let\@xp\subjclassname\csname subjclassname@#1\endcsname
  }%
}
 \makeatother
\usepackage{enumerate,amssymb,  mathrsfs,yhmath}

\newtheorem{theorem}{Theorem}[section]
\newtheorem{lemma}[theorem]{Lemma}
\newtheorem*{lemma*}{Lemma}
\newtheorem{proposition}[theorem]{Proposition}

\def\1ton{1,2,\ldots,n}

\def\ID{{\Bbb D}}

\usepackage{amssymb}
\usepackage{amsthm}
\usepackage{mathrsfs, amsfonts, amsmath}
\usepackage{graphicx}

\newcommand{\onto}{\xrightarrow[]{{}_{\!\!\textnormal{onto}\!\!}}}

\newcommand{\A}{\mathbb{A}}
\newcommand{\R}{\mathbb{R}}

\newcommand{\X}{\mathbb{X}}

\newcommand{\Y}{\mathbb{Y}}

\theoremstyle{definition}

\newtheorem{example}[theorem]{Example}

\theoremstyle{remark}
\newtheorem{remark}[theorem]{Remark}

\numberwithin{equation}{section}

\newcommand{\abs}[1]{\lvert#1\rvert}
\DeclareMathOperator{\sign}{sign} 
 \DeclareMathOperator{\id}{id}
\DeclareMathOperator{\Mod}{Mod}

\def\XXint#1#2#3{{\setbox0=\hbox{$#1{#2#3}{\int}$}
\vcenter{\hbox{$#2#3$}}\kern-.5\wd0}}

\def\ge{\geqslant}
\begin{document}

\title{Gaussian curvature of minimal graphs in $M\times \mathbb{R}$}  \subjclass{Primary 53A10 }


\author{David Kalaj }
\address{University of Montenegro, Faculty of Natural Sciences and
Mathematics, Cetinjski put b.b. 81000 Podgorica, Montenegro}
\email{davidk@ucg.ac.me}


\begin{abstract}
In this paper, we consider minimal graphs in the three-dimensional Riemannian manifold $M\times\mathbb{R}$. We mainly estimate the Gaussian curvature of such surfaces. We consider the minimal disks and minimal graphs bounded by two Jordan curves in parallel planes. The key to the proofs is the Weierstrass representation of those surfaces via $\wp-$harmonic mappings. We also prove some Schwarz lemma type results and some Heinz type results for harmonic mappings between geodesic disks in Riemannian surfaces.

\end{abstract}
\maketitle
\tableofcontents
\section{Introduction}

Assume that $\Omega$ is a domain in the complex plane. Moreover, assume that $\wp$ is a smooth positive metric defined in $\Omega$ with bounded Gauss
curvature $\mathcal{K}$ where
\begin{equation}\label{gaus}\mathcal{K}_\wp(z)=-\frac{\Delta \log \wp(z)}{\wp^2(z)}.\end{equation} We assume $\sup_{z\in \Omega} |\mathcal{K}_\wp(z)|<\infty$ and $\wp$ has a finite area defined by
$$\mathcal{A}(\wp)=\int_{\Omega}\wp^2(w) du dv, \ \ w=u+iv.$$ Then we define the Riemannian  manifold  $M=M_0 \times \mathbb{R}$, where $M_0= (\Omega, \wp)$, where the metric of $M$ is \begin{equation}\label{hbar}d\hbar^2 (z) =\wp^2(z) |dz|^2 + dt^2.\end{equation}

Let $N_0=(D, \mathbf{1})$ and $M_0=(\Omega,\wp)$ be
Riemann surfaces with metrics $\mathbf{1}$ and $\wp$, respectively, where $D$ and $\Omega$ are planar domains. Here $\mathbf{1}$ is the Euclidean metric. If
a mapping $f:N_0\to M_0,$ is $C^2$, then $f$ is said to be harmonic (to make difference with Euclidean harmonic we will sometimes say $\wp$-harmonic) if
\begin{equation}\label{el}
f_{z\overline z}+{(\log \wp^2)}_w\circ f\cdot  f_z\,f_{\bar z}=0,
\end{equation}
where $z$ and $w$ are the local parameters on $N_0$ and $M_0$
respectively. Also $f$ satisfies \eqref{el} if and only if its Hopf
differential
\begin{equation}\label{anal}
\mathrm{Hopf}(f)=\wp^2 \circ f f_z\overline{f_{\bar z}}
\end{equation} is a
holomorphic quadratic differential on $M$. Let $$|\partial
f|^2:=\wp^2(f(z))\left|\frac{\partial
f}{\partial z}\right|^2\text{ and }|\bar \partial
f|^2:={\wp^2(f(z))}\left|\frac{\partial
f}{\partial \bar z}\right|^2$$ where $\frac{\partial f}{\partial z}$
and $\frac{\partial f}{\partial \bar z}$ are standard complex
partial derivatives. The $\wp-$Jacobian is defined by
$$J(z,f):=|\partial f(z)|^2-|\partial \bar f(z)|^2.$$ If $f$ is sense preserving,
then its $\wp-$Jacobian is positive.   For $g:N_0 \mapsto M_0$ the $\wp-$
\emph{Dirichlet energy} is defined by
\begin{equation}\label{ener1} \mathcal{E}^\wp[g]=\int_{N_0} (|\partial
f|^2+ |\bar \partial
f|^2)dx\wedge dy=\int_{D} \wp^2(f(z))(|f_z|^2+|f_{\bar z}|^2)dxdy.
\end{equation}
Assume that the energy integral of $f$ is bounded. Then a
stationary point $f$ of $\mathcal{E}^\wp[g_t]$, i.e. the solution of the equation $\partial_t \mathcal{E}^\wp[g_t]|_{t=0}=0$, $g_t=g+t h$, $t\in(-\epsilon,\epsilon)$, where $h$ has a compact support in $D$, is a harmonic
mapping i.e. it satisfies equation \eqref{el}. A different variation produces the so-called Noether harmonic maps in Section~\ref{hopsec}. Recall that $N_0=(D,\mathbf{1})$. For the last
definition and some important properties of harmonic maps see
\cite[Chapter~8]{jost} and \cite{jostconf}.

It follows from the definition that, if $a$ is conformal
and $f$ is harmonic, then $f\circ a$ is harmonic. Moreover, if $b$ is
conformal, $b\circ f$ is also harmonic but with respect to
(possibly) an another metric $\wp_1$.

The Bochner formula \cite[p.~113]{jostconf} for a harmonic mapping  $h$ are $$\Delta \log |\partial h|^2 = -\mathcal{K}_\wp J(h)$$ and $$\Delta \log |\bar\partial h|^2 = \mathcal{K}_\wp J(h).$$

In this paper, we will prove some curvature estimates of a minimal graph $\Sigma$ in $M_0\times \mathbb{R}$. Firstly we will derive an analogous divergence formula for its non-parametric parametrization (Propositions~\ref{lemag} and \eqref{newmin}). Further, we consider disk-type minimal graphs and estimate the "central" curvature (Theorem~\ref{th1}). To prove Theorem~\ref{th1} we previously prove some Schwarz lemma type estimates of the dilatation of $\wp-$harmonic mappings, which we believe are interesting on their own (Theorem~\ref{teo1}). Also, we  prove a Schwarz lemma type estimate for the $\wp-$distance of the $\wp-$harmonic mapping (Theorem~\ref{teo2}). We also need some Heinz type estimate proved in the same theorem. The role of $\wp-$harmonic mappings in a minimal surface $\Sigma$ in $(\Omega,\wp)\times \mathbb{R}$ concerning Weierstrass representation  is analogous to the role of Euclidean harmonic mappings in Euclidean minimal surfaces (cf. \cite{ili}).


Moreover, minimizers of the $\wp-$energy of the diffeomorphic mappings between certain doubly connected domains produce minimal surfaces spanning two Jordan curves lying in parallel planes, provided the original annulus has a modulus bigger than the image annulus. This is done in Section~\ref{sec3}, where is made a connections with minimizers of $\wp-$energy  and minimal surfaces in Riemannian manifolds.

By using the results in Section~\ref{sec3}, we get the following Gaussian curvature at the "equator" of minimal annuli $\Sigma$ (see below Theorem~\ref{apo} for its reformulation ).
\begin{theorem}\label{exam1}

Assume that $\Sigma\subset \mathbb{R}^2\times \mathbb{R}$ is a double-connected minimal graph bounded by two Jordan curves $\gamma_1,\gamma_2$, lying in two horizontal  planes $\Pi_1, \Pi_2$ whose distance is $\mathbf{d}$. Let $\Pi$ be the equidistant plane between $\Pi_1,\Pi_2$. Then for every point $\zeta \in  \Sigma\bigcap \Pi$ we have $$|\mathcal{K}_\Sigma(\zeta)|\le \frac{\pi^2}{\mathbf{d}^2}.$$
Moreover, the tangential plane $T\Sigma_\zeta$ at a point $\zeta\in \Sigma\setminus (\gamma_1\cup \gamma_2)$ is never horizontal nor vertical.

\end{theorem}

At the end at Subsection~\ref{subsug}, some explicit curvature estimates for rotationally symmetric minimal annuli in Riemannian manifolds $(\A(R,1), \wp)\times \mathbb{R}$, are obtained. Here $\wp$ is also rotationally symmetric and $\A(R,1):=\{z: R<|z|<1\}$.

\subsection{Conformal minimal immersion}\label{subsec22}
A mapping $\chi(z) = (f(z), h(z)): \Omega \to (M,\wp)\times \mathbb{R}$ is  conformal minimal if and only if $(h_w)^2=-(\wp\circ h)^2 f_w\bar f_w$. Furthermore, the induced metric is given by $$ds^2=(\wp\circ h)^2(|f_z|+|f_{\bar z}|)^2|dz|^2.$$ Moreover, the unit normal $N$ at the tangent plane of the surface $\Sigma=\{\chi(z): z\in\Omega\}$ is given by the formula \begin{equation}\label{nn}N=\frac{1}{1+|F|^2}\left(\frac{2}{\wp}\Re F, \frac{2}{\wp}\Im F, |F|^2-1\right)$$ where $$F^2=-\frac{f_w}{\bar f_w}.\end{equation} For those facts we refer to \cite{ili}. We want to mention that a class of minima surfaces on the hyperbolic space $\mathbb{H}\times \mathbb{R}$ has been considered in \cite{tohoku} and \cite{nelli}.

\subsection{Minimal graphs}
For a surface $S=\{(u,v, \omega(u,v)): (u,v)\in \Omega\}\subset M=(\Omega, \wp)\times \mathbb{R}$, the $\wp$-area is defined by
\begin{equation}\label{inervar}\|\omega\|:=\mathrm{Area}_\wp(S)=\int_{\Omega} \sqrt{\wp^2(u,v)+|\nabla \omega(u,v)|^2}dudv.\end{equation}

By setting $\left(\frac{d}{dt} \|\omega_t\|\right)_{t=0}=0$, where $\omega_t$ is the outer variation of the form $\omega_t=\omega+t \varpi$,   $t\in(-\epsilon,\epsilon)$, where $\varpi$ has a compact support in $\Omega$, the following non-linear elliptic partial differential equation comes out
 \begin{equation}\label{newmin}\mathrm{div}\left(\frac{\nabla \omega}{\sqrt{1+\frac{|\nabla \omega(u,v)|^2}{\wp^2(u,v)}}}\right)=0.\end{equation}
Therefore, the graph $$\Sigma=\{(u,v, \omega(u,v)): (u,v)\in\Omega\}\subset M\times\mathbb{R}$$ is called minimal if $\omega$ solves the
the equation \eqref{newmin}. For $\wp\equiv 1$, the equation is the standard minimal surface equation.

We will relate the minimal graphs and conformal minimal immersions in Proposition~\ref{lemag}.

\section{Preliminary results from harmonic mappings and related topics}

\subsection{Schwarz lemma for harmonic mappings}

Let $o\in \Omega$ and let $D_R(o)=D(o,R)\subset (\Omega,\wp)$  be a geodesic disk of radius $R>0$ centered at $o$  and assume that the cut-locus of $o$ is outside of $D_R(o)$ and $2 R \kappa<\pi$, where $\kappa^2$ is an upper bound of the  Gaussian curvature  and define
$$\mathcal{H}_\kappa(o,R)=\{f:\mathbb{D}\to \Omega: f(\mathbb{D})\subset D_R(o), f \text{ is } \wp-\text{ harmonic }\}.$$
One of main results in \cite{hsw}  of Hildebrandt, Jost and Widman for harmonic maps between two-dimensional manifolds reads.
\begin{proposition}\cite[Theorem~3]{hsw}\label{propo}
The class $\mathcal{H}_\kappa(o,R)$ is H\"older continuous on every compact subset of the unit disk, i.e. for every $s<1$ there are constants $C>0$ and $0<\alpha<1$ depending on $s$, $R$ and $\kappa$ such that \begin{equation}\label{hol}d_\wp(f(z),f(w))\le  C|z-w|^\alpha\end{equation} for $|z|,|w|<s$, $f\in \mathcal{H}_\kappa(o,R)$.
\end{proposition}
Further we use the following theorem of the author
\begin{theorem}\label{lemjun}\cite{complex4} Let $$\mathcal{H}_\kappa(o,R)=\{f:\mathbb{D}\to \Omega: f(\mathbb{D})\subset D_R(o),\,2R\kappa<\pi, f \text{ is } \wp-\text{ harmonic }\}.$$
a) Then there is an absolutely continuous homeomorphism $$\phi_\wp=\phi_{o,R}:[0,1]\to[0,R]$$ convex in $\log r$, with $\phi'(1^-)>0$ and $0<\phi'(0^+)<\infty$, such that
$$d_\wp(f(z),f(0))\le\phi_\wp\left(\left|z\right|\right),$$ for $f\in \mathcal{H}_\kappa(o,R)$ with $f(0)=o$. Then $\ID\ni z\to \phi_\wp(|z|)$ is subharmonic.

b) Assume that the cut-locus of $o$ is outside of $D(o,R)$. Then there is an absolutely continuous homeomorphism $$\varphi_\wp=\varphi_{R}:[0,1]\to[0,2R]$$ convex in $\log r$, with $\varphi_\wp '(1^-)>0$ and $0<C'_\wp=\varphi_\wp'(0^+)<\infty$, such that
$$d_\wp(f(z),f(0))\le\varphi_\wp\left(\left|z\right|\right),$$ for $f\in \mathcal{H}_\kappa(o,R)$.  Moreover $$\wp(f(z))|\nabla f(z)|\le \frac{C'_\wp}{1-|z|^2}, \ z\in\ID.$$ Here $|\nabla f(z)|:=|f_z(z)|+|f_{\bar z}(z)|$.

\end{theorem}

Let $\mathfrak{M}_{\kappa}$ be the family of all smooth metrics $\wp$ in $\Omega$ so that $|\mathcal{K}_{\wp}|\le \kappa^2$, and so that the geodesic disk $D_1(o)=D_\wp(o,1)\subset \Omega$ is disjoint from its cut-locus and $2\kappa<\pi$. Let $\mathcal{H}_\kappa(o):=\{f:\mathbb{D}\to \Omega: f(\mathbb{D})\subset D_1(o),\,2\kappa<\pi, f \text{ is } \wp-\text{ harmonic for some } \wp\in \mathfrak{M}_{\kappa}  \text{ and } f(0)=o\}.$
Define \begin{equation}\label{pkappa}\phi_\kappa(r) =\sup\{\phi_\wp(r), \wp\in \mathfrak{M}_{\kappa} \}.\end{equation}

The continuity of $\phi_\kappa$ follows from Theorem~\ref{lemjun}, because the constant $C$ in \eqref{hol} depends on $\kappa$ only. To do so, observe that $$\phi_\wp(r_j)=\sup_{|z|=r_j, f\in H_\wp} d_\wp(f(z),f(0)),\  j=1,2,$$ and  assume without loss of generality that $\phi_\wp(r_2)\ge \phi_\wp(r_1)$. Then

\[\begin{split}|\phi_\wp(r_2)-\phi_\wp(r_1)|&=(\phi_\wp(r_2)-\phi_\wp(r_1))\\&=\sup_f d_\wp(f(r_1), o)-\sup_g d_\wp(g(r_2), o)\\&\le \sup_f\left[ d_\wp(f(r_1), o)- d_\wp(f(r_2),o)\right]\le \sup_f d_\wp(f(r_1), f(r_2))\\&\le C|r_1-r_2|^\alpha.\end{split}\]
Similarly we get $$|\phi_\kappa(r_2)-\phi_\kappa(r_1)|\le  C|r_1-r_2|^\alpha.$$
Further $\phi_\kappa$ is not constant, because $\phi_\kappa(0)=0$.

Then by well-known result for subharmonic functions, $\mathbb{D} \ni z\to \phi_\kappa(|z|)$ is a subharmonic absolutely continuous function on $r=|z|$ and   $\phi_\kappa(r)=\phi_{o}(r):[0,1]\to[0,1]$ is convex in $\log r$, with \begin{equation}\label{phi1}\phi_\kappa'(1^-)>0\end{equation} and
\begin{equation}\label{phi0}0<\phi_\kappa'(0^+)<\infty,\end{equation} such that
$$d_\wp(f(z),f(0))\le\phi_\kappa\left(\left|z\right|\right),$$ for $f\in \mathcal{H}(o)$ with $o=f(0)$ and $\wp\in \mathfrak{M}_{\kappa}$.

Notice that $\phi_\kappa$ is a nonconstant function. Since $\phi_\kappa $ is an
increasing convex function of $\log r$, it follows that for $r>s$
$$r\phi_\kappa'(r^+)\ge r\phi'(r^-)\ge s\phi_\kappa'(s^+)\ge s\phi_\kappa'(s^-)\ge 0.$$ Since it is non-constant, it satisfies in
particular that  $\phi'(1^-)>0$ which proves \eqref{phi1}. The relation \eqref{phi0} can be proved in the same way as the analogous relation for $\phi_\wp$ proved in \cite{complex4}.

Moreover, from Schwarz lemma for harmonic Euclidean harmonic functions (\cite[p.~76]{Duren2004}) we have $\phi_k(r)\ge \phi_0(r)=\frac{4}{\pi}\arctan r$. Since $\phi_k(1)=1$, we get $$\frac{1-\phi_k(r)}{1-r}\le \frac{1-\frac{4}{\pi}\arctan r}{1-r}.$$ By letting $r\to 1$ we get
\begin{equation}\label{derivat1}
\phi_\kappa'(1^-)\le \frac{2}{\pi}.
\end{equation}

So we have proved the first part of the following Schwarz lemma type estimate and Heinz type estimate for harmonic mappings.
\begin{theorem}\label{teo2}
Let $w=f(z)$ be a harmonic mapping of the unit disk  $\mathbb{D}$ into the geodesic disk $D(o,1)\subset
(\mathcal{M},\wp)$ of the Reimannian surface $(\mathcal{M},\wp)$ with
a Gaussian curvature $-\kappa^2\le \mathcal{K}_\wp\le \kappa^2$ and assume that $2\kappa<\pi$ and that the cut-locus of $o$ is outside of $D(o,1)$.  Let $\varrho$ be the distance function from the
fixed point $o\in \mathcal{M}$. Define $\rho(z)=\varrho(f(z))$.
Assume also that $o=f(0)$.

a) Then  we have the Schwarz lemma type inequality \begin{equation}\label{scwl}|\rho(z)|\le \phi_\kappa(|z|),\end{equation} where $\phi_\kappa:[0,1]\onto[0,1]$ is a bi-Lipschitz homeomorphism with $\phi_\kappa'(1^-)<2/\pi$.
Moreover, $$|\nabla \rho(0)|\le C_\kappa,$$ where $C_\kappa=\phi_\kappa'(0)\ge \frac{4}{\pi}$ and $\phi_\kappa$  depend on $\kappa$ only.

$a'$) If the cut-locus of $o$ is outside of $D(o,1)$ we have the inequality \begin{equation}\label{scwl2}d_\wp(f(z),f(0))|\le \varphi_\kappa(|z|), \ z\in\ID \end{equation} where $\varphi_\kappa:[0,1]\onto[0,2]$ is a homeomorphism locally Lipschitz in $[0,1)$.
Moreover, \begin{equation}\label{cpk}\wp(f(z))|\nabla f(z)|\le \frac{C'_\kappa}{1-|z|^2}, \ z\in\ID\end{equation} where $C'_\kappa=\varphi_\kappa'(0)<\infty$ and $\varphi_\kappa$  depend on $\kappa$ only.

b) If in addition the function $f:\mathbb{D}\onto D(o,1)$ is a diffeomorphism with $f(0)=o$, and $\wp$ has a non-negative  Gaussian curvature,  then \begin{equation}\label{maybe}|\partial f(z) |\ge c_\kappa>0, \ \ z\in \mathbb{D}.\end{equation} The function $\kappa\to c_\kappa$ is non-increasing.

\end{theorem}

\begin{remark}\label{vere}
a)
Instead  of $D(o,1)\subset(\mathcal{M},\wp)$, we can assume that $D(o,R)\subset(\mathcal{M},\wp)$, with $2\kappa R<\pi$. In that case we define $\wp_R(z) = R\cdot\wp(z)$. Then $D_{\wp}(o, R)=D_{\wp_R}(o,1)$ and \begin{equation}\label{cuca}\mathcal{K}_{\wp_{R}}=\frac{1}{R^2} \mathcal{K}_\wp<\frac{\pi^2}{\kappa^2}.\end{equation} In this case we get from \eqref{maybe} that \begin{equation}\label{maybe1}|\partial f(z) |\ge R c_\kappa, \ \ z\in \mathbb{D}.\end{equation}

A similar remark can be stated for $a')$.

If $\phi^R_\kappa(r)= R \phi_\kappa(r)$ is the corresponding homeomorphism between $[0,1]$ and $[0,R]$, then the inequality $\rho(z)\le \phi^R_\kappa(|z|)$ is sharp, provided that $f$ is a $\wp-$harmonic mapping of the unit disk onto $D_{\wp}(o, R)$ satisfying $f(0)=o$ under the constraint \eqref{cuca}.

b) Inequality \eqref{maybe} for $z=0$, with some absolute constant $c$ instead of $c_\kappa$ has been proved by Yau and Schoen in \cite{SchoenYau1997} for complete surfaces with positive curvature. We also want to observe that for $\kappa=0$, i.e. for flat metrics, we have $\phi_0(r)= \frac{4}{\pi}\arctan r$. Moreover, $C_0=4/\pi$. This follows from Heinz classical result \cite{he}. We want also to point out that the assumption $\kappa<\pi/2$ is essential. Namely it follows from the example in \cite{complex4}, that $C_{\kappa}\to\infty$ as $\kappa\to\pi/2$.

c) It follows from  Wan result \cite[Theorem~13]{tam} the following Heinz type sharp inequality $|\partial f(z)|\ge 1, z\in\ID$ (cf. \eqref{maybe} above) for every hyperbolic harmonic diffeomorphism of the unit disk onto itself. Observe that $\mathcal{K}_\lambda=-1$, for the hyperbolic metric $\lambda$.
\end{remark}

\begin{proof}[Proof of Theorem~\ref{teo2}]
{The  part a) of the theorem is already proved before its formulation.  The part $a'$) can be proved by imitating the proof of $a)$ by setting $\varphi_\kappa(r) =\sup\{d_\wp(f(z), f(0)): |z|=r, \wp\in\mathfrak{M}_\kappa, f:\mathbb{D}\to \Omega: f(\mathbb{D})\subset D_1(o),\,2\kappa<\pi, f \text{ is } \wp-\text{ harmonic } \}.$} Moreover \eqref{cpk} follows from \eqref{scwl2} by dividing it by $|z|$ and letting $z\to 0$.

To prove $b)$, assume first that $f$ has a smooth extension up to the boundary.
 We have by the triangle inequality that $$\frac{d_\wp(f(z),f(1))}{|z-1|}\ge \frac{\varrho(f(1))-\varrho(f(z))}{|z-1|}\ge \frac{1-\phi_\kappa(|z|)}{1-|z|}.$$ Then we obtain that $$\wp(f(e^{it}))|\partial_r f(re^{it})|_{r=1}\ge  2\tilde c_\kappa>0,$$
where  \begin{equation}\tilde c_\kappa=\frac{1}{2}\inf_{r\in[0,1)}\frac{1-\phi_\kappa(r)}{1-r}>0,\end{equation} because $$\lim_{r\to 1-0}\frac{1-\phi_\kappa(r)}{1-r}=\phi_\kappa'(1^-)>0.$$
Therefore \[\begin{split}\wp(f(e^{it}))|\partial_z f(e^{it})|&\ge \frac{1}{2}\wp(f(e^{it}))(|\partial_z f(e^{it})|+|\partial_{\bar z} f(e^{it})|)\\&\ge \frac{1}{2} \wp(f(e^{it}))|\partial_r f(re^{it})|_{r=1}\ge \tilde c_\kappa.\end{split}\]
Further since  $\mathcal{K}_\wp\ge 0$ and $J(z,f)=|\partial f(z)|^2-|\bar \partial f(z)|^2>0$, it follows that $\partial f(z)\neq 0$.  Then by Bochner's formula we have
$$\Delta \log \frac{1}{|\partial f|}=\mathcal{K}_\wp J_f\ge 0,$$ where $$|\partial f(z)|=\wp(f(z))|f_z(z)|.$$  From the maximum principle for subharmonic functions we have
$$\max_{|z|\le 1}\log \frac{1}{|\partial f|} =\max_{|z|= 1}\log \frac{1}{|\partial f|} .$$
So
$$\min_{|z|\le 1}{|\partial f|} =\min_{|z|= 1}{|\partial f|} \ge \tilde c_\kappa.$$
If $f$ has not a smooth extension up to the boundary then we use an approximation argument. We define the sequence $f_n(z) = \Psi^{-1}_n(f(\Phi_n(z)))$, where $\Phi_n: \mathbb{D}\onto f^{-1}(D(o, n/(n+1)))$, and $\Psi^{-1}_n: D(o, n/(n+1))\onto D(o,1)$ are conformal mappings so that
$\Phi_n(0)=0$, $\Psi_n(o)=o$, $\Phi_n'(0)>0$ and $\Psi_n'(o)>0$.  Then $f_n:\mathbb{D}\onto D(o, 1)$ is $\wp_n-$harmonic. Here $\wp_n(w)=\wp(\Psi_n(w))|\Psi_n'(w)|$, and $$\mathcal{K}_{\wp_n}(w)=\mathcal{K}_{\wp}(\Psi(w)),$$ so $0\le \mathcal{K}_{\wp_n}\le \kappa^2$ if  $0\le \mathcal{K}_{\wp}\le \kappa^2.$ Thus $$|\partial f_n(z)|\ge \tilde c_\kappa, \ z\in\mathbb{D}$$ and therefore $$|\partial f(z)|\ge \tilde c_\kappa, \ z\in\mathbb{D}$$  because $\Psi_n$ and $\Phi_n$ converges to identity in compacts. Now we define \begin{equation}\label{ck}c_\kappa=\inf\{|\partial f(z)|, z\in\ID, f\in \mathcal{H}_\kappa(o)\}.\end{equation} Then $c_\kappa\ge \tilde c_\kappa>0$.
Moreover, in view of \eqref{ck}, because $\mathcal{H}_{\kappa_1}(o)\subset \mathcal{H}_{\kappa_2}(o)$, for $\kappa_1<\kappa_2$, we obtain that $\kappa_1<\kappa_2\Rightarrow c_{\kappa_2}\le c_{\kappa_1}$.
\end{proof}
\subsection{Schwarz lemma for dilatation of harmonic mappings}
\begin{proposition}\label{prop2} \cite[ p.~10-11]{SchoenYau1997} Assume that $h$ is harmonic with respect to a non-vanishing $C^{2}$ metric $\rho$. Then the functions $|\partial h|$ and $|\bar \partial h|$ are identically zero or they have the isolated zeros with well defined orders. More precisely $|\partial h|=|z-z_0|^n u(z)$  and $|\bar\partial h|=|z-z_0|^m v(z)$ for continuous functions $u$ and  $v$ such that $u(z_0)\neq 0$ and $v(z_0)\neq 0$.
\end{proposition}
\begin{theorem}\label{teo1}  Let  $M=(\Omega,\wp)$ be a non-negatively curved Riemannian surface.
Assume that $f:\mathbb{D}\to M$ is a $\wp-$harmonic mapping that preserves the orientation. Let $m(z) = |{f_{\bar z}}/{f_z}|$. If $f_{\bar z}(0)=0$, then $$|m(z)|\le |z|^2,$$ or what is the same $$|f_{\bar z}(z)|\le |z|^2|f_z(z)|.$$
Further if  $\mu=\sqrt{m}$ then  \begin{equation}\label{numa}|\mu(z)|\le |z|,\text{ and }|\nabla \mu(0)|\le 1.\end{equation}
If $\mu(0)\neq 0$, $f:\ID\to D_1(o)$, and if $D_1(o)$ is disjoint from the cut-locus of $o$,  and $K_\wp\le \kappa^2<\pi^2/4$, then we have \begin{equation}\label{mumu}|\nabla \mu(0)|\le C_\kappa^\ast,\end{equation} for a constant $C_\kappa^\ast\ge 1$ depending on $\kappa$.

\end{theorem}

\begin{proof}
Prove first \eqref{mumu}.
Let $r\in(0,1)$, $P_r(z) = \frac{r^2-|z|^2}{|z-r|^2}$. By using the well-known  formula  we have
\begin{equation}\label{mpg}\mu(z) = h_r(z)+\frac{1}{r^2}\int_{r \ID} \log \left|\frac{z-w}{1-\frac{1}{r^2}z\overline{w}}\right| \Delta \mu (w)dudv, \ \ w=u+iv,\end{equation} where $h_r: r\ID \to \ID$ is the harmonic function defined by $$h_r(z)=\int_{0}^{2\pi}P_r(z e^{it})\mu[re^{it}]\frac{dt}{2\pi}.$$
Further $$\Delta \mu = K_\wp \wp^2(f(z)) (|f_z|^2 -|f_{\bar z}|^2).$$

From Schwarz lemma for Euclidean harmonic functions (\cite[p.~76]{Duren2004}) we obtain that $|\nabla h_r(0)|\le \frac{4}{r\pi}$.

On the other hand  $$\left|\nabla_z\log \left|\frac{z-w}{1-\frac{1}{r^2}z\overline{w}}\right|\right|_{z=0}=\frac{\sqrt{r^2-|w|^2}}{r^2|w|}.$$

Thus by  \eqref{cpk} we have $$|\nabla \mu(0)|\le \frac{4}{r\pi}+\kappa^2 (C'_\kappa)^2\frac{1}{r^2}\int_{r\ID} \frac{\sqrt{r^2-|w|^2}}{r^2|w|}\frac{1}{(1-|w|^2)^2}du dv.$$

Then for $r=1/2$ we obtain \begin{equation}\label{csk}|\nabla \mu(0)|\le C^*_\kappa:=\frac{8}{\pi}+ \pi\kappa^2 (C'_\kappa)^2   (5 \coth^{-1}(2)-2).\end{equation}

Now we improve the constant for the case $\mu(0)=0$.
 If $\mathcal{K}_\wp\ge 0$, and $f:\mathbb{D}\to \Omega$, then by Bochner formula we have $$\Delta \log \mu(z)= \frac{\mathcal{K}_{\wp}}{2} J\ge 0.$$ Since $$\mu=\sqrt{\frac{|f_{\bar z}|}{|f_z|}}=\frac{|\Phi(z)|}{\wp(f(z))|f_z|},$$ on account of Proposition~\ref{prop2}, it follows that also $\mu$ has well-defined order of the zero. So $\log \left(\mu(z)/|z|\right)$ is subharmonic.

Assume now that $\mu$ is continuous on  $\overline{\mathbb{D}}$. 
%
By the maximum principle $$\log (\mu(z)/|z|)\le 0.$$ Thus $$\mu(z)/|z| \le 1.$$ Therefore, $$|\mu(z)|\le |z|.$$

Further \begin{equation}\label{noth}\mu(u,v) = \mu(0,0) + \partial_x \mu(0,0)u+\partial_y \mu(0,0) + o(\sqrt{u^2+v^2}).\end{equation}
If $ \partial_x \mu(0,0)= \partial_y \mu(0,0)=0$ we have nothing to prove. If the last relation is not true then we choose $u= t\partial_x \mu(0,0)/|\nabla \mu(0,0)|$ and $v= t\partial_y \mu(0,0)/|\nabla \mu(0,0)|.$
By inserting $u$ and $v$ in \eqref{noth}, dividing by $t$  and letting $t\to 0$ we get $$|\nabla \mu(0,0)|\le 1.$$

If $\mu(z)/|z|$ is only upper-semicontinuous in $\overline{\mathbb{D}}$, then by \cite[Theorem~1.4]{hayman}, there exists a decreasing sequence of mappings $\mu_k(z)$ continuous in $\overline{\mathbb{D}}$ so that $$\lim_{k\to \infty} \mu_k(z) = \mu(z)/|z|.$$ Thus the proof can be reduced to the proof when $\mu$ is continuous up to the boundary.



\end{proof}

\section{Minimizers and minimal surfaces in Riemannian manifold}\label{sec3}

Assume  that $\X$ and $\Y$ are double connected domains and  that $\mathscr{W}_\wp^{1,2}(\X,\Y)$ is the class of mappings that belongs to  $\mathscr{W}_{\mathrm{loc}}^{1,2}$ and satisfy the inequality $$\int_{\X} \wp^2(f(z))(|f_z|^2+|f_{\bar z}|^2)dxdy +\int_{\X} \wp^2(f(z))|f(z)|^2 dxdy<\infty.$$

Assume that $\mathcal{H}^\wp(\X,\Y)\subset \mathscr{W}_\wp^{1,2}(\X,\Y)$ is the class of  homeomorphic mappings between $\X$ and $\Y$, that map the inner boundary onto inner boundary and outer boundary onto the outer boundary.

 Let $\overline{\mathcal{H}}^{\wp}(\X,\Y)$ be the closure of ${\mathcal{H}}^{\wp}(\X,\Y)$ in the strong topology of $\mathscr{W}_\wp^{1,2}(\X,\Y)$.

\subsection{Stationary mappings and Noether harmonic maps}\label{hopsec}

We call a mapping $h\in\overline{\mathcal{H}}^\wp(\X,\Y)$
\emph{stationary} if
\begin{equation}\label{stat}
\frac{d}{dt}\bigg|_{t=0}{\mathcal{E}^\wp}[h\circ \phi_t^{-1}]=0
\end{equation}
for every family of diffeomorphisms $t\to \phi_t\colon
\X\to\X$ which depend smoothly on the parameter $t\in\mathbb
R$ and satisfy
$\phi_0=\id$. The latter mean that the mapping $\X\times [0,\epsilon_0]\ni (z,t)\to \phi_t(z)\in \X $ is a smooth mapping for some $\epsilon_0>0$.
We now have.
\begin{lemma}\label{ctheory}\cite{calculus1}
Let $\X=\A(r,R)$ be a circular annulus, $0<r<R<\infty$, and
assume that $\Y$ is a  doubly connected domain. If $h\in
\overline{\mathcal{H}}^\wp(\X,\Y)$ is a stationary mapping, then
\begin{equation}\label{hopf1}\wp^2(h(z))
h_z\overline{h_{\bar z}} \equiv \frac{c}{z^2}\qquad \text{in }\X
\end{equation}
where $\mathbf{c}\in\R$ is a constant. Moreover, $\mathbf{c}\ge  0$ if $\mathrm{Mod}(\X)\le \mathrm{Mod}(\Y)$ and $\mathbf{c}< 0$ if $\mathrm{Mod}(\X)> \mathrm{Mod}(\Y)$. Further every stationary mapping that is a diffeomorphism is a minimizer.
\end{lemma}
Notice first that a change of variables $w=f(z)$ in~\eqref{ener1}
yields
\begin{equation}\label{ener2}
{\mathcal{E}^\wp}[f] = 2\int_{\X} \wp^2(f(z))J_f(z)\, dz +
4\int_{\X}\wp^2(f(z)) \abs{f_{\bar z}}^2dz\ge 2
\mathcal{A}(\wp)
\end{equation}
where $J_f$ is the Jacobian determinant and $\mathcal{A}(\wp)$ is
the area of $\Y$ and $dz:=dx\wedge dy$ is the area element
w.r. to Lebesgue measure on the complex plane. A conformal mapping
of $f:\X\onto\Y$; that is, a homeomorphic solution of the
Cauchy-Riemann system $ f_{\bar z}=0$, would be an obvious choice
for the minimizer of~\eqref{ener2}. For arbitrary multiply connected
domains there is no such mapping. In this way the $\wp-$harmonic mappings comes to the stage.
%

The following lemma makes a relation with minimizers and minimal surfaces in manifold $M=M_0\times \mathbb{R}$.



\begin{proposition}\label{lemag}
Assume that $f=a+\imath b:D\to \Omega$ is a diffeomorphism so that $$\wp^2(f(z)) f_z \bar f_z=\Phi^2(z),$$ where $\Phi$ is a holomorphic function in $D$. Assume that  for $z\in D$, the following function  $\Phi(z) = \int \Phi(z) dz$ is well-defined. Then for
\begin{equation}\label{hareq}
\omega(z) = 2\Re (\Phi(f^{-1}(z)))
\end{equation}
we have \begin{equation}\label{newmin1}\mathrm{div}\left(\frac{\nabla \omega}{\sqrt{1+\frac{|\nabla \omega(u,v)|^2}{\wp^2(u,v)}}}\right)=0.\end{equation}

In other words, $f$ defines the conformal parameters of the $\wp-$minimal graph $\Sigma=\{(u,v,\omega(u,v)), (u,v)\in \Omega\}$. The conformal parametrization $\chi: D\to (\Omega, \wp) \times \mathbb{R}$  is given by   $$\chi(z) = (\Re f,\Im f, 2\Re (\Phi(z))).$$
This means that $$|\chi_x|^2_{\hbar}-|\chi_y|^2_{\hbar}=\left<\chi_x, \chi_y\right>_{\hbar}=0,$$ where the metric $d\hbar^2$ is given in \eqref{hbar}.
\end{proposition}

\begin{proof}
Let $f=a+\imath b$.  Then
$$2\mathrm{Im}(f_z{ f_{ \bar z}})=-(a_x b_x +a_y b_y).$$
The  $\wp-$ harmonic equation \eqref{el} can be written
\begin{equation}\label{equ1}f_{z\bar z}+2\frac{\partial \wp (w)}{\wp(w)} \circ f \cdot f_z f_{\bar z}=0.\end{equation}
Equation \eqref{equ1} is equivalent with the following system of PDE
\begin{equation}\label{hhe}\Delta b=\left(\frac{4 \wp_u(f)}{\wp} \left(\mathrm{Im}(f_z f_{\bar z})\right)+\frac{\wp_v(f)}{\wp} \left(4\mathrm{Re}(f_z f_{\bar z}\right)\right)\end{equation}

\begin{equation}\label{gge}\Delta a=\left(\frac{4\wp_v(f)}{\wp} \left(\mathrm{Im}(f_z f_{\bar z})\right)+\frac{\wp_u(f)}{\wp} \left(-4\mathrm{Re}(f_z f_{\bar z})\right)\right).\end{equation}

From \eqref{hareq} we obtain that the function $H(z)=\omega(f(z))$ is a harmonic function, so its Laplacian vanishes and this can be written in the form
\[\begin{split}(b_y^2 &+b_x^2)\omega_{vv}(w) +2 (a_x b_x+a_y b_y )\omega_{uv}(w)+\left(a_y^2+a_x^2\right) \omega_{uu}(w)\\&+\omega_{u}(w) \left(a_{yy}+a_{xx}\right)+\omega_{v}(w) \left(b_{xx}+b_{yy}\right)=0\end{split}\]
which can be written as
\begin{equation}\label{vendos}\begin{split}(|f_z-&\bar f_z|^2)\omega_{vv}(w) -4\mathrm{Im}(f_z{ f_{ \bar z}})\omega_{uv}(w)+\left(|f_z+\bar f_z|^2\right) \omega_{uu}(w) \\&+\omega_{u}(w) \left(4 \frac{\wp_v(f)}{\wp} \left(\mathrm{Im}(f_z f_{\bar z})\right)+\frac{\wp_u(f)}{\wp} \left(-4\mathrm{Re}(f_z f_{\bar z})\right)\right)
\\&+\omega_{v}(w) \left(\frac{4 \wp_u(f)}{\wp} \left(\mathrm{Im}(f_z f_{\bar z})\right)+\frac{\wp_v(f)}{\wp} \left(4\mathrm{Re}(f_z f_{\bar z}\right)\right)=0.
\end{split}\end{equation}
From
$$\wp^2(f(z)) f_z \bar f_z=\Psi^2(z),$$
we have \begin{equation}\label{fpf}\bar f_z=\frac{\Psi^2(z)}{f_z \wp^2(f(z))}.\end{equation}
Let $$\Phi=\int \Psi(z) dz.$$
By differentiating the equation $$\omega(f(z))=\psi(z)=2\Re (\Phi(z))$$ we have
\begin{equation}\label{epp}\partial \omega f_z + \bar\partial \omega \bar f_z=\Phi'(z)=\Psi(z).\end{equation}
Now recall
$f= a+\imath b$
and so
$f_z+\bar f_z=a_x-\imath a_y$
and
$f_z-\bar f_z=-b_y+\imath b_x$.
Further from \eqref{fpf} and \eqref{epp} we obatin
$$ f_z =\frac{\Psi(z)\left(\sqrt{\wp^2+|\nabla \omega|^2}-\wp\right)}{2  \partial \omega \wp}$$ and
$$ \bar f_{ z }=-\frac{\Psi(z)\left(\sqrt{\wp^2+|\nabla \omega|^2}+\wp\right)}{ 2 \bar \partial \omega \wp}$$
and so
$$f_z  f_{\bar z} = \frac{-|\Psi(z)|^2}{4\wp^2 (\partial \omega)^2},$$
and
\begin{equation}\label{prima}4\mathrm{Im}(f_z  f_{\bar z}) =\frac{-8|\Psi(z)|^2}{\wp^2}\frac{\omega_u \omega_v}{ |\nabla \omega|^4}\end{equation}
and
\begin{equation}\label{seconda}4\mathrm{Re}(f_z  f_{\bar z}) =\frac{-4|\Psi(z)|^2}{|z|^2\wp^2}\frac{\omega_u^2- \omega_v^2}{ |\nabla \omega|^4}.\end{equation}
If
$$\omega_u-\imath \omega_v = R e^{\imath s},$$
then  \begin{equation}\begin{split}\label{fpz}|f_z+f_{\bar z}|^2= \frac{|\Psi(z)|^2 \left(\wp^2+ \omega_v^2\right)}{\wp^2 R^2}\end{split}\end{equation}
and \begin{equation}\label{fmz}\begin{split}|f_z-f_{\bar z}|^2&=\frac{|\Psi(z)|^2  \left(\wp^2+\omega_u^2\right)}{\wp^2  R^2}.\end{split}
\end{equation}
Inserting \eqref{fpz},\eqref{fmz}, \eqref{prima} and \eqref{seconda} to \eqref{vendos} we get
\[\begin{split}&(\wp^3+\wp \omega_{v}(w)^2) \omega_{uu}(w)+(\wp^3+\wp \omega_{u}(w)^2) \omega_{vv}(w)-2\omega_{u}(w)  \wp \omega_{v}(w) \omega_{uv}(w)\\&+(\wp_u (w) \omega_u(w)+\wp_{v}(w) \omega_{v}(w)) \left(\omega_{v}(w)^2+\omega_{u}(w)^2\right)=0\end{split}\]
which can be compactly written as \eqref{newmin}.

\end{proof}
\begin{example}
If $f: \A(1/R,R) \onto \Omega$ is a minimizer of the $\wp$ energy, then, by a result of Kalaj and Lamel in \cite{kalajlamel}, it defines a catenoid type minimal surface $$X_1=\{(f_1(z), h_1(z)): h_1(z)= 2\sqrt{|\mathbf{c}|}\log {|z|}: z\in \A(1/R,R)\}$$ or a helicoid type minimal surface
$$X_2=\{(f_1(z), h_2(z)): h_2(z)= 2\sqrt{|\mathbf{c}|}\,\mathrm{arg}z: z\in \A(1/R,R)\}.$$

Those two minimal surfaces are conjugate to each other in the notation of \cite{Laurent}. Namely their Hopf differentials satisfy the relation $$\mathrm{Hopf}(f_1)=-\mathrm{Hopf}(f_2)=\frac{|\mathbf{c}|}{z^2}.$$
\end{example}
Now Proposition~\ref{lemag} has this specific form for minimal surfaces from \cite{kalajlamel}.
\begin{proposition}\label{lemlem}
Assume that $f=a+\imath b:\Omega\to D$ is a diffeomorphism so that $$\wp^2(f(z)) f_z \bar f_z=\frac{\mathbf{c}}{z^2},$$
with $\mathbf{c}\in\mathbb{R}$. Assume that  for $z\in\Omega$, the following function is well-defined
\begin{equation}\label{hareq1}
\omega(z) = \left\{
              \begin{array}{ll}
                2\sqrt{-\mathbf{c}}\log \frac{1}{|f^{-1}(z)|}, & \hbox{if $\mathbf{c}< 0$;} \\
                \hbox{$2\sqrt{\mathbf{c}}\,\mathrm{arg}(f^{-1}(z))$}, & \hbox{if $\mathbf{c}> 0$.}
              \end{array}
            \right.
\end{equation}

Then $\omega$ satisfies the minimal surface equation \eqref{newmin}.
\end{proposition}

\begin{example}\label{exam}
If $\wp=\rho(|w|)$, then from \cite{jlondon}, a radial $\wp-$harmonic mapping $f$ of the annulus $\A(r,1)$ onto $\A(R,1)$, $r,R<1$ is given implicitly by the formula $$f^{-1}(w) = e^{\imath \tau}\exp\int_1^s\frac{\rho(y)}{\sqrt{4\mathbf{c}+y^2\rho^2(y)}}dy, w=se^{\imath \tau},$$ where $$\max_{R\le |y|\le 1} 4\mathbf{c}+y^2\rho^2(y)\ge 0,$$ and $$r=\exp\int_1^R\frac{\rho(y)}{\sqrt{4\mathbf{c}+y^2\rho^2(y)}}dy.$$
Then for $\mathbf{c}<0$ $$\omega(z)=\sqrt{-\mathbf{c}}\log\frac{1}{|f^{-1}(z)|}=\int_{|z|}^1\frac{\rho(y)}{\sqrt{4\mathbf{c}+y^2\rho^2(y)}}dy.$$
Thus after straightforward computation we get
$$\left(\frac{\nabla \omega}{\sqrt{1+\frac{|\nabla \omega(x,y)|^2}{\rho^2(x,y)}}}\right)=\frac{2\sqrt{-\mathbf{c}}}{\bar z}$$ and so that $$\mathrm{div}\frac{2\sqrt{-\mathbf{c}}}{\bar z}=2\sqrt{-\mathbf{c}}\left(\frac{2}{|z|^2}-\frac{2}{|z|^2}\right)=0.$$
\end{example}

\section{Curvature estimates of minimal graphs}
In \cite{meekshel} Meeks and White considered minimal surfaces in $\mathbb{R}^2\times \mathbb{R}$ bounded by convex curves in parallel planes and give some improvement of some earlier results on this topic by Shiffman \cite{shiffman}. This result has its counterpart to the minimizers \cite{koh1} (a result of  Ngin-Tee Koh).  In this section we consider some curvature estimates of minimal graphs bounded by Jordan curves in parallel planes in $\mathbb{R}^2\times \mathbb{R}$ and also in $(\Omega,\wp)\times \mathbb{R}$.

In the following proposition, we obtain a  connection between minimizers of energy and minimal graphs bounded by Jordan curves in parallel planes.
\begin{proposition}\label{p1}Assume that $\Sigma$ is a minimal graph over a double connected domain $\Omega$ in the complex plane. Assume also that  $\Sigma$ spans two planar curves $\gamma_1$ and $\gamma_2$ that belongs to two horizontal planes whose distance is $\mathbf{d}$.
\begin{enumerate}
\item
Then there is a conformal minimal parametrisation  $\chi: \A(1/R,R)\to \Sigma$ where $R>1$. In this case $$\Mod(\Sigma)=\Mod(\A(1/R,R))=\log R^2.$$
\item The mapping $f(z) = \chi_1(z)+\imath \chi_2(z)$ is a harmonic diffeomorphism between $A=\A(1/R,R)$ and $\Omega$ with $\mathrm{Mod}(A)>\mathrm{Mod}(\Omega)$.
\item
 Then there is a constant $\mathbf{c}<0$ so that $\Sigma$ is the graphic of a real function $u:\Omega\to \mathbb{R}$, of the form $u(z) =C+ 2\sqrt{-\mathbf{c}}\log |f^{-1}(z)|$, which satisfies the equation \eqref{newmin1} above where $\wp\equiv 1$. Moreover, $f_z\cdot \bar f_z = \frac{\mathbf{c}}{z^2},$ where \begin{equation}\label{ccc}\mathbf{c}=-\frac{\mathbf{d}^2}{4\mathrm{Mod}(\Sigma)^2}.\end{equation}
\end{enumerate}
\end{proposition}

\begin{remark}
The question arises for which Jordan curves $\gamma_1$ and $\gamma_2$ lying to different parallel planes, there is a minimal graph spanning them. A more general question is whether exist a minimal surface spanning $\gamma_1$ and $\gamma_2$. By a classical result of Courant, if $m(\gamma_i)$ is the area of the domain inside $\gamma_i$,  $i=1,2$ and $m(\gamma_1,\gamma_2)$ is the least area of all surfaces spanning $\gamma_1$ and $\gamma_2$, then the condition $m(\gamma_1,\gamma_2)< m(\gamma_1)+m(\gamma_2)$ implies that there exist a minimal doubly connected surface spanning $\gamma_1$ and $\gamma_2$. If the projection of such surface in the plane $P$ parallel to both $\gamma_1, \gamma_2$ is a doubly connected domain $\Omega$, then we conjecture that there is a curve $\gamma_0\subset \Sigma$ parallel to $\gamma_1, \gamma_2$ so that $\Sigma(\gamma_1,\gamma_0)$ and $\Sigma (\gamma_2,\gamma_0)$ are minimal graphs, one of which, denote it by $\Sigma'$,  has a projection the whole domain $\Omega$. In that case $$\Mod(\Omega)<\Mod(\Sigma')<\Mod(\Sigma).$$ Moreover, such a setting gives rise to a harmonic diffeomorphism $f$ of the annulus $\A(1/R, R)$ onto $\Omega$, where

\begin{equation}\label{ineq}\Mod(\A(1/R,R))>\Mod(\Omega).\end{equation}

We mention here that when \begin{equation}\label{opineq}\Mod(\A(1/R,R))\le \Mod(\Omega)\end{equation} the existence of harmonic diffeomorphisms  between $\A(1/R,R)$ and $\Omega$  has been established in \cite{calculus1}.
We refer to \cite{koli} where  Kovalev and Li considered the existence of harmonic diffeomorphisms under the condition \eqref{ineq}, under some smoothness conditions on the boundaries and some constraint on their conformal modulus. The general problem of determining the conditions under which there is a harmonic diffeomorphism with given constraints  \eqref{ineq} on the annuli remains open.
\end{remark}
\begin{proof}
The item a) follows from the fact that $\Sigma$ is smooth double connected surface and therefore it exist a conformal mapping $\chi=(\chi_1,\chi_2, \chi_3)$ of an annulus $\A(1/R,R)$ onto $\Sigma$ by virtue of theorem \cite[Theorem~3.1]{jostconf} see also \cite[Lemma~2.2]{meekshel}. Moreover, all coordinates of $\chi$ are harmonic. Since $\Sigma$ is a graph of a real function $u$ defined in a double connected domain $\Omega$, it follows that $\Sigma =\{(x,y,u(x,y)):(x,y)\in\Omega\}$. Therefore the mapping $f(z) =\chi_1(z) +\imath \chi_2(z)$ is a diffeomorphism as the composition of the following two diffeomorphisms $\chi$ and $(x,y, u(x,y))\to (x,y)$.
Further $\chi_3(z)=u(f(z))$ and $\chi_3(1/R e^{\imath t})=\mathbf{h}$ and $\chi_3(R e^{\imath t})=\mathbf{h}+\mathbf{d}$. Since $\chi_3$ is harmonic in $\A(1/R,R)$ it has the following form \begin{equation}\label{v3h}\chi_3(z)=h(z)=\frac{\mathbf{d}}{2}+\mathbf{h}+\frac{\mathbf{d} }{\log \left[R^2\right]}\log |z|.\end{equation}
Thus by conformal condition we have $$(f_z(z)+\bar{f}_z(z))^2+\left(\imath (f_z(z)-\bar{f}_z(z))\right)^2+ \left(\frac{\mathbf{d} }{z \log \left[R^2\right]}\right)^2=0.$$
Therefore, $$f_z(z)\bar{f}_z(z)=-\left(\frac{\mathbf{d} }{2z \log \left[R^2\right]}\right)^2=\frac{\mathbf{c}}{z^2},$$ where $$\mathbf{c}=-\frac{\mathbf{d}^2}{4\mathrm{Mod}(\Sigma)^2}.$$
Now Lemma~\ref{ctheory} implies that if $\mathbf{c}<0$ then $\Mod A\ge \Mod\Omega$. But $\Mod A=\Mod\Omega$ is not possible, because in that case $f$ is conformal, and thus $\mathbf{c}=0$ which is also not possible. So $\Mod A>\Mod\Omega$. The item (3) follows from Lemma~\ref{lemlem}.
\end{proof}
Similarly we can prove the following proposition
\begin{proposition}\label{p2}Assume that $\Sigma\subset (\Omega,\wp)\times \mathbb{R}$ is a minimal graph over a double connected domain $\Omega$ in the complex plane. Assume also that  $\Sigma$ spans two planar curves $\gamma_1$ and $\gamma_2$ that belongs to two horizontal planes whose distance is $\mathbf{d}$.
\begin{enumerate}
\item
Then there is a conformal minimal parametrisation $\chi: \A(1/R,R)\to \Sigma$ where $R>1$. In this case $$\Mod(\Sigma)=\Mod(\A(1/R,R))=\log R^2.$$
\item The mapping $f(z) = \chi_1(z)+\imath \chi_2(z)$ is a $\wp-$harmonic diffeomorphism between $A=\A(1/R,R)$ and $\Omega$ with $\mathrm{Mod}(A)>\mathrm{Mod}(\Omega)$.
\item
 Then there is a constant $\mathbf{c}<0$ so that $\Sigma$ is the graphic of a real function $u:\Omega\to \mathbb{R}$, of the form $u(z) =C+ 2\sqrt{-\mathbf{c}}\log |f^{-1}(z)|$ satisfying  equation \eqref{newmin1} above. Moreover, $\wp^2(f(z))f_z\cdot \bar f_z = \frac{\mathbf{c}}{z^2},$ where $\mathbf{c}$ is given by \eqref{ccc}. \end{enumerate}
\end{proposition}



\subsection{Gaussian curvature}
Assume that $\Sigma=\{(f(z), h(z)): z\in\Omega\}$ is a minimal surface in the manifold $M=(\Omega, \wp) \times \mathbb{R}$.
From a corresponding formula in \cite{Laurent} we get
$$ds^2_\Sigma=\wp(f(z))^2(|f_z|+|f_{\bar z}|)^2|dz|^2=\frac{1}{4}\left(|\mu|+|\mu|^{-1}\right)^2 |\eta|^2|dz|^2=\lambda|dz|^2,$$
where
$$\eta=\pm 2\rho(f(z))\sqrt{f_z\bar f_z}$$ and
$$\mu = \sqrt{\left|\frac{\bar f_{z}}{ f_z}\right|}.$$
Let
$m=\mu^2.$
Bochner formula implies that
$$\Delta \log m = \mathcal{K}_\wp J_f,$$ where
$$J(f)=|\partial_\wp f|^2 - |\bar \partial_\wp f|^2=\wp(f(z))^2(|f_z|^2-|f_{\bar z}|^2).$$
Then the Gaussian curvature of $\Sigma$ at $\zeta=\chi(z)=(f(z), h(z))$ can be expressed as
$$\mathcal{K}=\mathcal{K}_\Sigma(\zeta)=-\frac{\vartriangle \log \lambda(z)}{2\lambda(z)}.$$
Thus
\begin{equation}\label{curva}\mathcal{K}=\frac{-\Delta \log \lambda}{2\lambda}=\frac{-2 A+X \mu^2  \left(1-\mu^4 \right)}{2\lambda  \mu^2 (1+\mu^2)^2}\end{equation}
where
$$A=|\nabla m|^2$$
and
$$X=2\Delta \log m =2 \mathcal{K}_{\wp} J(f).$$
So $$\mathcal{K}=\frac{- A}{\lambda  \mu^2 (1+\mu^2)^2}+\frac{\mathcal{K}_{\wp} J(f)  \left(1-\mu^4 \right)}{\lambda  (1+\mu^2)^2}.$$
Observe that, if $\wp$ has negative Gaussian curvature, then $\Sigma$ is also negatively curved minimal surface.
Furthermore
$$0\le \frac{J(f)  \left(1-\mu^4 \right)}{\lambda  (1+\mu^2)^2}=\frac{ \left(1-\mu^2\right)^2 }{ \left(1+\mu^2\right)^3} \le 1 .$$
Further $$ \frac{- A}{\lambda  \mu^2 (1+\mu^2)^2}=\frac{- B }{\wp^2(f(z))(|f_z|+|f_{\bar z}|)^2 \left(1+\mu^2\right)^4},$$ where \begin{equation}\label{B}B=4|\nabla \mu|^2.\end{equation}

Thus we get \begin{equation}\label{kal}|\mathcal{K}|\le \max\left\{\left|\frac{- B }{\wp^2(f(z))(|f_z|+|f_{\bar z}|)^2 }\right|,\left|\frac{- B }{\wp^2(f(z))(|f_z|+|f_{\bar z}|)^2 }+\mathcal{K}_\wp\right|\right\}.\end{equation}

By \cite[Proposition~4]{ili}, see \eqref{nn} above, we know that a conformal parametrization $\chi(z) = (f(z), h(z))$ of a minimal surface $\Sigma\subset M\times \mathbb{R}$ defines the unit normal at a given point. Moreover, if $\bar\partial f(0)=0$, then the unit normal $N$ at a point $\chi(0)=o$  is  $N=(0,0,1)$ and it defines a horizontal tangential plane. Now we formulate the following Finn-Osserman type result (\cite{FinnOsserman1964}):

\begin{theorem}\label{th1}
 Let $D_\wp(o, R)\subset M_0= (\Omega,\wp)$ be a  geodesic disk so that $0\le \mathcal{K}_\wp\le \kappa^2<\frac{\pi^2}{4R^2}$. Assume further that $D_\wp(o,R)$ is disjoint from the cut-locus of $o$.  Assume that $\Sigma$ is a minimal surface over  $D_\wp(o, R)$ in the Riemannian manifold $M_0\times \mathbb{R}$. Then at the point $O$ above $o$, we have the following Gaussian curvature estimate of $\Sigma$
$$|\mathcal{K}_\Sigma(O)|\le\frac{4C_\kappa^\ast}{R^2c_k^2},$$ where $c_\kappa$ is defined in \eqref{ck} and $C_\kappa^\ast\ge 1$ is defined in \eqref{csk}.
If the tangential plane at $O$ is horizontal, then we can choose $C_\kappa^\ast=1$.
\end{theorem}

\begin{proof}

Observe that, from Remark~\ref{vere} we have that  $c_\kappa\le 2/\pi$. Further, having in mind Theorem~\ref{teo1} we have $|B|\le 4C_\kappa^\ast$, where $B$ is defined in \eqref{B}. Now by \eqref{kal} and Theorem~\ref{teo2} (equation \eqref{maybe1}) we get $|\mathcal{K}|\le \frac{4C_\kappa^\ast}{R^2c_k^2}$, because $\frac{4}{c_k^2}> \pi^2$ and $\mathcal{K}_\wp<\pi^2/(4R^2)$.
\end{proof}
\subsection{The curvature estimate of minimal graphs between parallel planes}
In this case, as it is discussed in Propositions~\ref{p1} and \ref{p2} we have
$$\eta=\pm 2\rho(f(z))\sqrt{f_z\bar f_z}=\pm 2\imath \sqrt{\frac{\mathbf{c}}{z^2}}.$$
For our surface $$ \mathcal{K} = \frac{- A |z|^2}{|c| \left(1+\mu^2\right)^4}+\frac{\mathcal{K}_{\wp} J(f)  \left(1-\mu^4 \right)}{2\lambda  (1+\mu^2)^2}.$$
If $\wp\equiv 1$, then $\mathcal{K}_\wp=0$ and $$\frac{-\Delta \log \lambda}{\lambda}=\frac{- A}{2\lambda  \mu^2 (1+\mu^2)^2} .$$
In this case $m=|\mathbf{q}(z)|^2$, where $\mathbf{q}$ is a holomorphic function so that $|\mathbf{q}(z)|<1$,  and so $$|\nabla m|^2=4|\mathbf{q}(z)|^2\cdot |\mathbf{q}'(z)|^2.$$ Moreover,
$$\lambda =(1+|\mathbf{q}|^2)^2 |\mathbf{p}|^2.$$
In this case we get $$\mathcal{K}=-\frac{4|\mathbf{q}'|^2}{|\mathbf{p}|^2(1+|\mathbf{q}|^2)^4}.$$
We also have for $\omega=\mathbf{q}^2$, (\cite[p.~184]{Duren2004}) \begin{equation}\label{exp}\mathcal{K}=-\frac{|\omega'|^2}{|f_z \bar f_z|(1+|\omega|)^4}.\end{equation}

Now we have the make the following equivalent formulation of Theorem~\ref{exam1}
\begin{theorem}\label{apo}

Assume that $\Sigma\subset \mathbb{R}^2\times \mathbb{R}$ is a double-connected minimal graph bounded by two Jordan curves in parallel planes $z=\pm \mathbf{d}/2.$ Then for every point $\zeta \in  \Sigma\bigcap \mathbb{R}^2\times \{0\}$ we have $$|\mathcal{K}(\zeta)|\le \frac{\pi^2}{\mathbf{d}^2}.$$

Moreover, the tangential plane $T\Sigma_\zeta$ at a point $\zeta\in \Sigma\setminus \mathbb{R}^2\times \{-\mathbf{d}/2,\mathbf{d/2}\} $ is never horizontal nor vertical.

\end{theorem}

\begin{proof}
Assume that $\Sigma$ is defined over a double connected domain $\Omega$ and assume that $\chi(z)=(f(z), h(z)): \A(1/R, R)\to \Sigma$ is its conformal minimal parameterisation. By \eqref{v3h} $h$ is given by  \begin{equation}\label{specific}h(z)=\frac{\mathbf{d} }{\log \left[R^2\right]}\log |z|.\end{equation}

 Let $\omega(z) = (\mathbf{q}(z))^2$.
From \eqref{exp} we have \begin{equation}\label{curv}\mathcal{K}(\chi(z)))=-\frac{|\omega'(z)|^2}{|f_z\bar f_z|(1+|\omega|)^4}=-\frac{|\omega'(z)|^2|z|^2}{|\mathbf{c}|(1+|\omega|)^4}.\end{equation}
Schwarz-Pick inequality for  analytic functions of the annulus $\A(1/R, R)$ into the unit disk, $R>1$ can be formulated as (\cite{minda})
\begin{equation}\label{idb} \left|\omega'(z)\right|<\frac{\pi  \left(1-\left|\omega(z)\right|^2\right) \sec\left[\frac{\pi  \log |z|}{2 \log R}\right]}{4|z| \log R}.\end{equation}
(Notice that: For $\omega(z) = z^2/R^2$ the above inequality is sharp for $s=|z|\to R$. For $\omega(z) = 1/(R^2z^2)$ it is sharp for $s=|z|\to 1/R$.)

So $$|\mathcal{K}| \le \frac{|\omega'(z)|^2|z|^2}{|\mathbf{c}|(1+|\omega|)^4}\le  \left(\pi  \frac{(1-|\omega(z)|)}{1+|\omega(z)|}\frac{\sec\left[ \frac{\pi \log |z|}{2\log R}\right]}{4\log R}\right)^2\frac{1}{|\mathbf{c}|}.$$
So for every $\zeta \in \Sigma$ we have
$$|\mathcal{K}(\zeta)| \le \frac{|\omega'(z)|^2|z|^2}{|\mathbf{c}|(1+|\omega|)^4}\le  \left(\pi  \frac{\sec\left[ \frac{\pi \log |z|}{2\log R}\right]}{4\log R}\right)^2\frac{1}{|\mathbf{c}|}.$$
So for $|z|=1$, on account of \eqref{specific}, we have that $\zeta=\chi(z)\in \Sigma\bigcap \mathbb{R}^2\times \{0\}$. For such $z$  we have \begin{equation}\label{wol}|\mathcal{K}(\zeta)|\le  \left( \frac{\pi }{4\log R}\right)^2\frac{1}{|\mathbf{c}|}.\end{equation}
From \eqref{wol} and \eqref{ccc} we obtain
\[\begin{split}|\mathcal{K}(\zeta)|&\le  \left( \frac{\pi }{4\log R}\right)^2\frac{4\mathrm{Mod}(\Sigma)^2}{\mathbf{d}^2}\\&=\left( \frac{\pi }{2\log R^2 }\right)^2\frac{4\mathrm{Mod}(\Sigma)^2}{\mathbf{d}^2}\\&=\frac{\pi^2}{\mathbf{d}^2}.\end{split}\]

To prove the last assertion of the theorem, observe that  $\omega'(z) = 2 \mathbf{q}'(z) \mathbf{q}'(z)$.
Then the unit normal at $\zeta=\chi(z)\in \Sigma$, in view of \cite[p.~169]{Duren2004} or \cite[eq.~2.81]{fors}  is given by $$\mathbf{n}_{\zeta}=-\frac{1}{1+|\mathbf{q}(z)|^2}(2\Im \mathbf{q}(z), 2\Re \mathbf{q}(z), -1+|\mathbf{q}(z)|^2).$$
Further $T\Sigma_\zeta$ is horizontal if and only if $\mathbf{q}(z)=0$ at $\zeta=\chi(z)$. In this case $\omega(z)=0$, but this is impossible because $f_z \bar f_z=c/z^2$ and so $f_z\neq 0$ and $f_{\bar z}\neq 0$. Now the tangent plane $T\Sigma_\zeta$ is vertical at $\zeta=\chi(z)$, if and only if  $|\mathbf{q}(z)|=1$, but then the maximum principle  this implies that $\mathbf{q}\equiv e^{\imath \tau}$. This is not possible.
\end{proof}

\subsection{The curvature of rotationally symmetric minimal surface}\label{subsug} We continue the notation of Example~\eqref{exam}.
Assume that $p: [r,1]\to [R,1]$ is a function defined by $$ s=\exp\int_1^{p(s)}\frac{\rho(y)}{\sqrt{4\mathbf{c}+y^2\rho^2(y)}}dy$$
Define for $\mathbf{c}<0$: $$\chi(z) = (p(s)\cos \theta, p(s) \sin \theta, 2\sqrt{-\mathbf{c}}\log \frac{1}{s})$$ and for $\mathbf{c}>0$: $$\chi(z) = (p(s)\cos \theta, p(s) \sin \theta, 2\sqrt{\mathbf{c}} \theta),$$ where $z=se^{\imath \theta}$. Then $\chi(z)$ is a conformal parametrisations of a corresponding minimal surface $\Sigma_{\mathbf{c}}$. Observe that for $\rho\equiv 1$, i.e. for $\wp-$being an Euclidean metric, the surface $\Sigma_{\mathbf{c}}$ is a catenoid for $\mathbf{c}<0$ and it is a helicoid for $\mathbf{c}>0$.

If we set  $f(z) = p(s) e^{\imath \theta}$, $z=s e^{\imath \theta}$, then $f$ is $\rho-$harmonic mapping between $\A(r,1)$ and $\A(R,1).$ Since $$p(s)^2-(s p'(s))^2= -\frac{4 \mathbf{c}}{\rho(p(s))^2},$$  we get for $\mathbf{c}<0$ that $$\mu_f=\frac{|f_{\bar z}|}{|f_z|}=\frac{p(s)-sp'(s)}{p(s)+p'(s)}.$$ By having in mind the equation $$|\eta|=\frac{|4\mathbf{c}|}{|z|^2},$$ and using the formula \eqref{curva}, after long but straightforward calculations
we get
\[\begin{split} \mathcal{K}=\mathcal{K}_{\mathbf{c}}(\chi(z))&=-\frac{\sign(\mathbf{c})}{t^4 \rho^6} \bigg(8 \mathbf{c} t^2 \rho'^2+\rho^2 \left(4 \mathbf{c}+t^4 \rho'^2\right)\\&+4 \mathbf{c} t \rho \left(\rho'-t \rho''\right)-t^3 \rho^3 \left(\rho'+t \rho''\right)\bigg)    \end{split}\]
where  $t=p(s)$ and $\wp=\rho(t)$,
where $\mathbf{c}$ is a constant with
$$4\mathbf{c}+\min_{t\in[R,1]}t^2\rho^2(t)\ge 0.$$
Then for $\mathbf{c}<0$ $$\mathcal{K}=\frac{4 \mathbf{c} (\rho+t\rho')^2}{t^4 \rho^6}+\mathcal{K}_\wp \left(1+\frac{4 \mathbf{c}}{t^2 \rho^2}\right) .$$
For $\mathbf{c}>0$ we have $$ \mathcal{K} = -\frac{4 \mathbf{c} \left(\rho+t \rho'\right)^2}{\left(4 \mathbf{c} \rho+t^2 \rho^3\right)^2}+\mathcal{K}_\wp\frac{t^2\rho^2}{4\mathbf{c}+t^2\rho^2}.$$
For $\wp$ being an Euclidean metric we have $$\mathcal{K}=\left\{
                                                  \begin{array}{ll}
                                                    \frac{4 \mathbf{c}}{t^4}, & \hbox{if $\mathbf{c}<0$;} \\
                                                    -\frac{4 \mathbf{c}}{\left(4 \mathbf{c} +t^2\right)^2}, & \hbox{if $\mathbf{c}>0$.}
                                                  \end{array}
                                                \right.$$

{\bibliographystyle{abbrv} 
}

\end{document}